\documentclass[12pt]{article}
\usepackage{amsfonts, amsmath, amssymb, amsgen, amsthm, amscd, latexsym}
\usepackage[all]{xy}

\def\Id{\mathop{\rm Id}\nolimits}

\def\Hom{\mathop{\rm Hom}\nolimits}

\def\Cb{{\mathbb C}}

\def\Nb{{\mathbb N}}

\def\Zb{{\mathbb Z}}

\def\Hc{{\cal H}}

\def\Ic{{\cal I}}

\def\Kc{{\cal K}}

\def\Cc{{\cal C}}

\def\d{\delta}
\def\D{\Delta}
\def\g{\gamma}
\def\G{\Gamma}
\def\lb{\lambda}

\def\om{\omega}
\def\Om{\Omega}
\def\s{\sigma}

\def\ve{\varepsilon}
\def\vp{\varphi}
\def\nr{\natural}

\def\0b{\bf 0}

\def\ot{\otimes}

\def\ra{\rightarrow}

\def\al{>\hspace{-4pt}\vartriangleleft}

\def\hd{\overset{\ra}{\partial}}
\def \vd{\uparrow\hspace{-4pt}\partial}
\def\hs{\overset{\ra}{\sigma}}
\def \vs{\uparrow\hspace{-4pt}\sigma}
\def\hta{\overset{\ra}{\tau}}
\def \vta{\uparrow\hspace{-4pt}\tau}

\def\p{\partial}

\def\0D{\Delta^{(0)}}
\def\1D{\Delta^{(1)}}

\def\wt{\widetilde}
\def\td{\tilde}
\def\otb{{  +\hspace{-9pt}\ot }}

\newcommand{\FD}{\mathfrak{D}}

\newtheorem{theorem}{Theorem}[section]

\newtheorem{proposition}[theorem]{Proposition}
\newtheorem{lemma}[theorem]{Lemma}
\newtheorem{corollary}[theorem]{Corollary}

\newtheorem{example}[theorem]{Example}

\def\build#1_#2^#3{\mathrel{
\mathop{\kern 0pt#1}\limits_{#2}^{#3}}}
\newcommand{\ps}[1]{~\hspace{-4pt}_{^{(#1)}}}
\newcommand{\ns}[1]{~\hspace{-4pt}_{_{{<#1>}}}}
\newcommand{\sns}[1]{~\hspace{-4pt}_{_{{<\overline{#1}>}}}}

\def\odots{\ot\dots\ot}

\numberwithin{equation}{section}

\def\d{\delta}

\def\g{\gamma}

\def\lb{\lambda}
\def\om{\omega}
\def\s{\sigma}

\def\ve{\varepsilon}

\def\vp{\varphi}

\def\D{\Delta}
\def\G{\Gamma}

\def\Om{\Omega}

\def\ot{\otimes}
\def\part{\partial}

\def\ra{\rightarrow}

\def\text{\hbox}

\def\ot{\otimes}

\def\ra{\rightarrow}

\def\wt{\widetilde}

\def\Hom{\mathop{\rm Hom}\nolimits}

\def\Id{\mathop{\rm Id}\nolimits}

\def\build#1_#2^#3{\mathrel{
\mathop{\kern 0pt#1}\limits_{#2}^{#3}}}

\numberwithin{equation}{section}
\begin{document}

\title{ Cup products in Hopf cyclic cohomology via cyclic modules I}
\author{ Bahram Rangipour\\ \hline
 \footnotesize{ Department of Mathematics and
Statistics,}\\
\footnotesize{University of New Brunswick,}\\
\footnotesize{ Fredericton, NB,  CANADA, E3B 5A3} \\ \hline
\tiny{Email address:~~{bahram@unb.ca}} }
\date{}
\maketitle

\begin{abstract}
\noindent  This is the first one in  a series of  two papers on the
continuation of our study in cup products in  Hopf cyclic
cohomology. In this note we construct cyclic cocycles of algebras
out of  Hopf cyclic cocycles of algebras and coalgebras. In the next
paper we consider producing Hopf cyclic cocycle from ``equivariant"
Hopf cyclic cocycles. Our approach in both situations  is based on
(co)cyclic modules and bi(co)cyclic modules together with
Eilenberg-Zilber theorem which is different from the old definition
of cup products defined in \cite{kr1}  via traces and cotraces on DG
algebras and coalgebras.

\end{abstract}

\section{ Introduction and some preliminaries }
Hopf cyclic cohomology was  invented by Connes and Moscovici
 as part of their fundamental work of  computing the class of the index of the
hypoelliptic signature operator \cite{cm2}. The decidedly nontrivial
idea was to show that the index cocyclic is in the range of a
characteristic map.

\smallskip

  Hopf cyclic cohomology was vastly generalized to
study  Hopf-(co)module (co)algebras and coefficients(partially in
\cite{kr} and comprehensively in \cite{hkrs2, hkrs1}); later on Hopf
cyclic cohomology was generalized to encompass the category of
bialgebra-(co)module (co)algebras \cite{k}. In \cite{hkrs2} it was
conjectured that any characteristic map as above is just a component
of a cup product in Hopf cyclic cohomology.  In \cite{kr1} the
author and M. Khalkhali proved the existence of the cup products
defined via traces and cotraces over DG algebras and coalgebras. As
an intermediate step, characteristic maps via higher traces of
Crainic \cite{cr} and Gorokhovsky \cite{gor} can be thought of as
cup products with trivial coefficients. In \cite{cm7} it is
beautifully disclosed  how the idea of cup product is applicable in
case that (or even in general) the algebra  under question possesses
no invariant trace; as a replacement  one takes advantage of an
invariant cyclic cocycle to realize Hopf cyclic cocycles as cyclic
cocycles on the algebra.

\bigskip

In these  notes, we  define and study the above cup products in a
very straightforward way by using a method we learned from
\cite{mr}. The fact that the  equivariant property of cocycles
yields that the produced cocycle to be  well-defined on the
convolution and crossed product algebra is nontrivial. This prompt
us  to do  more research on this cup products \cite{Ran}. This
products also has to be analyzed from the category and
representation theory point of view. The merit of our definitions is
their simplicity and lack of dependence  on the algebra or coalgebra
structures, since we use a Hopf twisting map the whole procedure
should work to a great extent for arbitrary twisting maps. As one
knows the cyclic cohomology of Hopf algebras is defined as the
cyclic cohomology of a canonical cocyclic module associated  to the
Hopf algebra; one uses this fact to produce Hopf cyclic cocycles by
exploiting  `` equivariant" Hopf cyclic cocycles \cite{Ran}.

\bigskip

There are at least eight  kinds of cup products defined  on Hopf
cyclic cohomology but only two of them so far  are applied in NCG,
the reason could be their  lack of  classic  and/or geometric
counterparts. For the first product, one starts with Hopf cyclic
cocycles over an algebra and a coalgebra with coefficients in a SAYD
module. To define the cup product one needs the  coalgebra to act on
the algebra. The next step is to construct, via a  twisting map, the
cup product as cocycle over the convolution algebra. But one knows
that any cocycle over the convolution algebra is automatically a
cocycle over the algebra. This cup product generalizes the
characteristic map of Connes-Moscovici \cite{cm2}. Ingredients  for
the second cup product are  cyclic cocycles on a module algebra over
a Hopf algebra  and on  a comodule algebra over the same Hopf
algebra. Out of the two Hopf cyclic cocycles  one produces a cyclic
cocycle over the crossed product algebra. This cup product
generalizes the ordinary cup product in cyclic cohomology of
algebras as defined by Connes \cite{cng}.

\bigskip

For the  reader convenience,  we  briefly   recall the definition of
Hopf cyclic cohomology of coalgebras and algebras under the symmetry
of Hopf algebras and with coefficients in  stable anti
Yetter-Drinfeld (SAYD) modules \cite{cm2,hkrs1, hkrs2}. In this note
$\Hc$ is a Hopf algebra, $\mu, \eta, \D,\ve,$ and $ S$ be its
product, unite, coproduct, counit and antipode, which is also
supposed  invertible, respectively. We use the Sweedler's notation
for coproduct, i.e., $\D(h)=h\ps{1}\ot h\ps{2}$. Let $C$ be a
$\Hc$-module coalgebra, that is a coalgebra endowed with an action,
say from left,
 of $\Hc$ such that its comultiplication and counit are $\Hc$-linear, i.e,
\begin{align}
\D(hc)=h\ps{1}c\ps{1}\ot h\ps{2}c\ps{2}, \quad \ve(hc)=\ve(h)\ve(c).
\end{align}
As the coefficients in this cohomology theory the notion of  SAYD
module is defined in \cite{hkrs1} and recalled as follows. It is
said that a right module $M$ which is also  a left comodule is a
right-left SAYD module over a Hopf algebra  $\Hc$ if it satisfies
the following conditions for any $h\in \Hc$, and $m\in M$.
\begin{align}\label{SAYD1}
&m\sns{0}m\sns{-1}=m\\\label{SAYD2}
 &(mh)\sns{-1}\ot (mh)\sns{0}=
S(h\ps{3})m\sns{-1}h\ps{1}\ot m\sns{0}h\ps{2},
\end{align}
where the coaction of $\Hc$ is denoted by $\D_M(m)=m\sns{-1}\ot
m\sns{0}$.

\medskip
Having the  datum $(\Hc, C, M)$, one defines \cite{hkrs2} cocyclic
module $\{C^n_\Hc(C,M), \p_i,\s_j,\tau \}_{n\ge 0}$   as follows.

 \begin{equation*}
 \Cc^n:= C^n_\Hc(C,M)=M\ot_\Hc C^{\ot n+1}, \quad n\ge 0,
 \end{equation*}
 with the following cocyclic structure,
\begin{align}
&\p_i:\Cc^n\ra \Cc^{n+1}, & 0\le i\le n+1\\
& \s_j: \Cc^n\ra \Cc^{n-1}, & 0\le j\le n-1,\\
& \tau:\Cc^n\ra \Cc^n,
\end{align}

defined explicitly as follows, where we abbreviate $\td c=c^0\odots
c^n$,
 \begin{align}\label{CCM1}
& \p_i(m\ot_\Hc \td c )=m\ot_\Hc c^0\odots \Delta(c_i)\odots c^n,
\\\label{CCM2}
 &\p_{n+1}(m\ot_\Hc \td c)=m\sns{0}\ot_\Hc c^0\ps{2}\ot c^1\odots c^n\ot
  m\sns{-1}c^0\ps{1},\\\label{CCM3}
 &\sigma_i(m\ot_\Hc \td c)=m\ot_\Hc c^0\odots \epsilon(c^{i+1})\odots c^n,
 \\\label{CCM4}
 &\tau(m\ot_\Hc \td c)=m\sns{0}\ot_\Hc c^1\odots c^n\ot m\sns{-1}c^0.
 \end{align}
It is checked   \cite{hkrs2} that $ \p_i,\s_j,$ and $\tau$ satisfy
the following identities, which are  recalled  from \cite{cng} as
the definition of cocyclic module.
\begin{equation}\label{ds}
\p_j  \p_i = \p_i  \p_{j-1}, \, \, i < j  , \qquad \s_j \s_i = \s_i
\s_{j+1},  \, \,  i \leq j
\end{equation}
\begin{equation} \label{sd}
\s_j  \p_i = \left\{ \begin{matrix} \p_i  \s_{j-1} \hfill &i < j
\hfill \cr 1_n \hfill &\hbox{if} \ i=j \ \hbox{or} \ i = j+1 \cr
\p_{i-1}  \s_j \hfill &i > j+1  ;  \hfill \cr
\end{matrix} \right.
\end{equation}
\begin{eqnarray} \label{ci}
\tau_n  \p_i  = \p_{i-1}  \tau_{n-1} ,
 \quad && 1 \leq i \leq n ,  \quad \tau_n  \p_0 =
\p_n \\ \label{cj} \tau_n  \s_i = \s_{i-1} \tau_{n+1} , \quad &&
1 \leq i \leq n , \quad \tau_n  \s_0 = \s_n  \tau_{n+1}^2 \\
\label{ce} \tau_n^{n+1} &=& 1_n  \, .
\end{eqnarray}

As the motivating example  of the above theory  one recovers the
cyclic complex of a Hopf algebra $\Hc$ endowed with a modular pair
in involution (MPI) $(\d,\s)$, which we recall it  here from
\cite{cm2}. The character $\d$ is an algebra map $\Hc\ra \Cb$, and
the group-like element $\s\in \Hc$ is a coalgebra map $\Cb\ra \Hc$,
i.e. $\s:=\s(1)$ satisfies $\D(\s)=\s\ot \s$. The pair $(\d, \s)$ is
called MPI if $\d(\s)=1$, and  $\td{ S}_\d=Ad \s$, where the twisted
antipode $\td S_\d$ is defined by
\begin{equation}
\td{ S_\d}(h)=(\d\ast S)(h)=\d(h\ps{1})S(h\ps{2}).
\end{equation}

One  knows that  $\Hc$ is left $\Hc$-module coalgebra via left
multiplication. On the other hand if one lets $M=^\s\Cb_\d$ to be
the ground field $\Cb$ endowed with the left $\Hc$ coaction  via
$\s$ and right $\Hc$ action  via the character $\d$, then  its
checked \cite{hkrs2} that $(\d,\s)$ is a MPI if and only if
$^\s\Cb_\d$ is a SAYD. Thanks to the multiplication and the antipode
of $\Hc$ one identifies $C_\Hc(\Hc, M)$ with $M\ot \Hc^{\ot n}$ via
the following map,
\begin{align}
&\Ic: M\ot_\Hc \Hc^{\ot (n+1)}\ra M\ot \Hc^{\ot n},\\
&\Ic(m\ot_\Hc h^0\odots h^n)=mh^0\ps{1}\ot S(h\ps{2})\cdot(h^1\odots
h^n).
\end{align}
As a result one simplifies $\p_i$, $\s_j$, and $\tau$, in this case,
and recovers the original definition of the cyclic cohomology of
Hopf algebras \cite{cm2}.

\begin{eqnarray} \label{Ans}
\p_0 (h^1 \ot \ldots \ot h^{n-1}) &=& 1 \ot h^1 \ot \ldots \ot
h^{n-1} , \\  \nonumber \p_j (h^1 \ot \ldots \ot h^{n-1}) &=& h^1
\ot \ldots \ot \D h^j \ot \ldots \ot h^{n-1} , \qquad 1 \leq j \leq
n-1
\\  \nonumber
\p_n (h^1 \ot \ldots \ot h^{n-1}) &=& h^1 \ot \ldots \ot h^{n-1} \ot
\s, \\ \nonumber
 \s_i (h^1 \ot \ldots \ot h^{n+1}) &=& h^1 \ot \ldots \ot \ve
(h^{i+1}) \ot \ldots \ot h^{n+1} , \qquad 0 \leq i \leq n \, , \\
\nonumber
   \tau_n (h^1 \ot \ldots \ot h^n) &=& (\D^{p-1}  \wt S (h^1)) \cdot h^2
\ot \ldots \ot h^n \ot \s .
\end{eqnarray}

Similarly  an algebra  which is $\Hc$-module and its algebra
structure is $\Hc$-linear is called $\Hc$-module algebra. Let $A$ be
a  $\Hc$-module algebra and then one endows $M\ot A^{\ot n+1}$ with
the diagonal action of $\Hc$ and forms $C^n_\Hc(A,M)=\Hom_\Hc(M\ot
A^{\ot n+1}, \Cb)$ as the space of  $\Hc$-linear maps. It is checked
that the following defines a cocyclic module structure on
$C^n(A,M)$.
\begin{align*}
&(\p_i \vp)(m\ot \td{a}) = \vp(m\ot a^0\ot \dots
\ot a^i a^{i+1}\ot\dots \ot a^{n+1}), &0 \leq i <n,\\
&  (\p_{n+1}\vp)(m\ot \td{a}) =\vp(m\sns{0}\ot
(S^{-1}(m\sns{-1})a^{n+1})a^0\ot a^1\ot\dots \ot a^{n}),\\
&(\s_i\vp)(m\ot \td a) = \vp(m\ot a^0 \ot \dots
\ot a^i\ot 1 \ot \dots \ot a^{n-1}), &0\le i\le n-1,\\
&(\tau\vp)(m\ot \td a) = \vp(m\sns{0}\ot (S^{-1}(m\sns{-1})a^n)\ot
a^0\ot \dots \ot a^{n-1}),
\end{align*}

The cyclic cohomology of this cocyclic module is denoted by
$HC^\ast_\Hc(A,M)$.

\medskip

An algebra is called  $\Hc$-comodule coalgebra if  it is a $\Hc$
comodule and its coalgebra structure are $\Hc$ colinear. Similar to
the other case, one  defines  $^\Hc C^n(A,M)$ to be the space of all
colinear maps from $A^{\ot n+1}$ to $M$. One checks that the
following defines a cocyclic module structure on $^\Hc C^n(A,M)$.

\begin{align*}
&(\p_i\vp)(\td a)=\vp( a^0\ot \dots
 \ot a^i a^{i+1}\ot\dots \ot a^{n+1}), &\quad 0\le i< n,\\
 &(\p_{n+1}\vp)(\td a)= \vp(a^{n+1}\ns{0}a^0\ot a^1\dots \ot
a^{n-1}\ot a^{n})a^{n+1}\ns{-1},\\
& (\s_i \vp)( \td a)= \vp(a^0 \ot \dots \ot a^i\ot
 1 \ot \dots \ot a^{n-1}), & \quad 0\le i\le n-1,\\
&(\tau\vp )(a^0\ot\dots\ot a^n)=\vp( a^n\ns{0}\ot
  a^0\ot\dots \ot a^{n-1}\ot a^{n-1})a^{n}\ns{-1}.
\end{align*}
The cyclic cohomology of this cocyclic module is denoted  by $^\Hc
HC^\ast(A,M)$.
 For completeness, we record below the
bi-complex
$$\, (C C^{*, *} (C, \Hc,  M), \, b , \, B )$$
that computes the Hopf cyclic cohomology of a coalgebra $C$ with
coefficients in a SAYD module $M$ under the symmetry of a Hopf
algebra $\Hc$:
\begin{equation*} \label{CbB}
CC^{p, q} (C, \Hc; M)  \, = \, \left\{
\begin{matrix}  C^{q-p}_\Hc (C, M ) \, , \quad q \geq p \, , \cr
  0 \, ,  \quad \qquad  \qquad q <  p \, , \end{matrix}   \right.
\end{equation*}

    the operator
\begin{equation*}  \nonumber
b: C^{n}_\Hc (C,M) \ra C^{n+1} _\Hc(C,M), \qquad b =
\sum_{i=0}^{n+1} (-1)^i \p_i \,
\end{equation*}

and  the $B$-operator $B:C^{n}_\Hc (C,M) \ra C^{n-1}_\Hc (C,M)$ is
defined by the formula
\begin{equation} \nonumber
  B = A \circ B_{0} \, ,  \quad n \geq 0 \, ,
\end{equation}
where
\begin{equation*}
B_{0}= \s_{n-1}\tau(1-(-1)^n\tau)
\end{equation*}
 and
\begin{equation} \nonumber
A = 1 + \lb  + \cdots +  \lb^n \,  , \qquad \text{with} \qquad \lb =
(-1)^{n-1} \tau_n \, .
\end{equation}

The groups $\, \{H C^n  (\Hc; \delta,\sigma)\}_{n \in \Nb} \,$ are
computed from the first quadrant total complex $\, ( TC^{*} (\Hc;
\delta,\sigma),\, b+B ) \,$,
\begin{equation*}
    TC^{n}(\Hc; \delta,\sigma) \,  = \, \sum_{k=0}^{n} \, C C^{k, n-k}
     (\Hc; \delta,\sigma) \, ,
\end{equation*}
and the periodic groups $\, \{H P^i  (\Hc; \delta,\sigma)\}_{i \in
\Zb/2} \,$ are computed from the full total complex $\,(T P^{*}(\Hc;
\delta,\sigma),\,  b+B ) \,$,
\begin{equation*}
      T P^{i}(\Hc; \delta,\sigma) \,  = \, \sum_{k \in \Zb} \,
       C C^{k, i-k} (\Hc; \delta,\sigma) \, .
\end{equation*}

Let $(C^n, \delta_i, \sigma_i, \tau_n)$ and $(C'^n, \delta_i,
\sigma_i, \tau_n)$ be two cocyclic objects in the category of vector
spaces. Their product is the cocyclic object $((C\times C')^n,
\delta_i, \sigma_i, \tau_n)$ with $(C\times C')^n = C^n \otimes
C'^n$ and $\delta_i =\delta_i \otimes \delta_i)$, $\sigma_i
=\sigma_i \otimes \sigma_i$ and $\tau_n =\tau_n \otimes \tau_n$.

Their tensor product is the bicocyclic module $C \otimes C'$ defined
by $ (C \otimes C')^{m, n}=C^m \otimes C^n$ with horizontal and
vertical structure borrowed from $C$ and $C'$ respectively.
Eilenberg-Zilber states that the Cyclic cyclic cohomology of mixed
complexes  $C\times C'$ and  $Tot(C\ot C')$ are the same via the the
shuffle map \cite{lo}.
\section{  Module algebras paired with   module coalgebras }
Let $\Hc$ be a Hopf algebra,  $A$ be a $\Hc$-module algebra and
 $(\d,\s)$ be
 a modular pair in involution on $\Hc$.
Connes and Moscovici  \cite{cm2} showed  that the following defines
 a map of cocyclic modules
\begin{align}
&\chi: \Hc^\nr_{(\d,\s)}\ra C^\ast(A),\\\notag &\chi(h^1\odots
h^n)(a^0\odots a^n)=\tau(a^0h^1(a^1)\dots h^n(a^n)).
\end{align}
Here   $\tau: A\ra \Cb$ is a $\d$-invariant  \;\;$\s$-trace, i.e.
for all $a, b$ and $h$
\begin{align}
\tau(ha)=\d(h)\tau(a),\\
\tau(ab)=\tau(b\s a).
\end{align}
The above map then induces the following characteristic map on the
level
 of cohomologies:
\begin{equation}
\chi:HC^n_{(\d,\s)}(\Hc)\ra HC^n(A).
\end{equation}
   Hopf cyclic cohomology and  SAYD (stable anti-Yetter-Drinfeld)  modules,
   generalize  cyclic cohomology of Hopf algebras and MPI
   (modular pair in involution),  respectively.
   Now  a $\d$-invariant
$\s$-trace  is exactly a closed cyclic cocycle in $C^0_\Hc(A,
^\s\Cb_\d)$. These facts prompted  us  in \cite{hkrs2} to conjecture
that there should exist a generalization of characteristic map  as a
pairing between Hopf cyclic cohomology of module algebras and module
coalgebras:
\begin{equation}\label{hkrs}
HC^n_\Hc(A,M)\ot HC^m_\Hc(C,M)\ra HC^{n+m}(A,M),
\end{equation}
where  $M$ is a left-right SAYD module over $H$ and $C$ is  a $\Hc$
module coalgebra acting on $A$ in the sense that there is a map
\begin{equation}\label{0}
C\ot A\ra A,
\end{equation}
such that for any $h\in\Hc$, any $c\in C$ and any $a,b\in A$ one has
\begin{align}\label{1}
 &(hc)a= h(ca)\\\label{2}
 &c(ab)=(c\ps{1}a)(c\ps{2}b)\\\label{3}
&c(1)=\epsilon(c)1
\end{align}

Although there is a proof of the above conjecture in \cite{kr1}, we
would like to give a more direct proof  based on theory of cyclic
modules instead of  traces on DG algebras.
 The advantage of this
 new view is not only its simplicity but also its efficiency
  which enables one
  to use the precise expression of these cup products as it is shown  in \cite{Ran}.

\bigskip
 One constructs a very useful convolution algebra $B=\Hom_\Hc(C,A)$,
which is the algebra of all $\Hc$-linear  maps from $A$ to $C$.
 The unit of this algebra is given by $\eta\circ\epsilon$,
  where $\eta:\Cb\ra A$ is the unit of $A$.
  The multiplication of $f,q\in B$ is given by
\begin{equation}
(f\ast g)(c)=f(c\ps{1})g(c\ps{2})
\end{equation}
Now consider  two cocyclic modules
$$(C^\ast_\Hc(A,M), \d_i,\s_j,t), \quad \text{and}\quad
(C^\ast_\Hc(C,M), d_i,s_j,\tau)$$ and let us  make a new bicocyclic
module by just tensoring these two ones.  The bicocyclic module has
in bidegree $(p,q)$
$$C^{p,q}:=\Hom_\Hc(M\ot A^{\ot p+1},\Cb)\otb (M\ot_\Hc C^{\ot q+1}),$$ with
horizontal structure $\hd_i=d_i\ot\Id$, $\hs_j=s_j\ot\Id$, and
$\hta=t\ot \Id$ and vertical structure $\vd_i=\d_i\ot\Id$,
$\vs_j=\s_j\ot\Id$, and $\vta=\tau\ot \Id$. Obviously
$(C^{n,m},\hd,\hs,\hta,\vd,\vs,\vta)$  defines a bicocyclic module,
where $\otb:=\ot$ for which  we use
 it to distinguish between  $\Hom_\Hc(M\ot A^{\ot p+1},\Cb)$ and  $(M\ot_\Hc
C^{\ot q+1})$.

Now let us define  the  map
\begin{align}\label{acpsi}
&\Psi_c:C^{n,n}\ra \Hom(B^{\ot n+1},\Cb),\\\notag & \Psi_c(\phi\otb
m\ot c^0\odots c^n)(f^0\odots f^n)=\phi(m\ot f^0(c^0)\odots
f^n(c^n)),
\end{align}
which is obviously well defined due to  the facts that $f$ is
$\Hc$-linear, $\phi$ is  equivariant  and \eqref{1} holds.

  \begin{proposition}
The map $\Psi_c$ defines a  cyclic map between  the diagonal of
$C^{\ast,\ast}$ and the cocyclic module $C^\ast(B)$.
\end{proposition}
\begin{proof}
We have to show that $\Psi$ commutes  with the cyclic structures of
the  two sides. Indeed we just check it for the first  face and
cyclic operator and leave the rest to the reader. We have
\begin{align*}
&\Psi_c(\hd_0\vd_0(\phi\otb m\ot c^0\odots c^n))(f^0\odots
f^{n+1})=\\
&\Psi_c(d_0\phi\otb
\d_0(m\ot c^0\odots c^n))(f^0\odots f^{n+1})=\\
&\phi(m\ot f^0(c^0\ps{1})f^1(c^0\ps{2})\ot  f^2(c^1)\odots
f^{n+1}(c^n)=\\
&\phi(m\ot (f^0\ast f^1)(c^0)\ot  f^2(c^1)\odots
f^{n+1}(c^n))=\\
&(d_0\Psi_c(\phi\otb m\ot \td c))(f^0\odots f^{n+1}). \\
\end{align*}
\begin{align*}
&\Psi_c(\hta\vta(\phi\otb m\ot c^0\odots c^n))(a^0\odots
a^{n})= \\
&\Psi_c(t\phi\otb
\tau(m\ot c^0\odots c^n))(f^0\odots f^{n})=\\
&t\phi(m\sns{0} f^0(c^1)\odots f^{n-1}(c^n))\ot
m\sns{-1}f^n(c^0))=\\
&\phi(m\sns{0}\ot S^{-1}(m\sns{-1})m\sns{-2}f^n(c^0)\ot
f^0(c^1)\odots f^{n-1}(c^n))=\\
&\phi(m\ot f^n(c^0)\ot
f^0(c^1)\odots  f^{n-1}(c^n))=\\
&(t\Psi_c(\phi\otb m\ot c^0\odots c^n))(f^0\odots f^{n}).
\end{align*}
\end{proof}
One can take advantage of   properties  \eqref{1},\eqref{2},
 and \eqref{3} to prove  that
there exists  a natural unital algebra map $\nr:  A\ra
Hom_\Hc(A,C)$,
  explicitly defined by
$\nr(a)(c)=c(a)$.  As a result, one obtains a cyclic map
 $\nr: C^\ast( B,\Cb)\ra C^\ast(A,\Cb)$.  One then composes
$\nr$ with $\Psi_c$ to get a cyclic map
$$\Psi=\nr\circ\Psi_c: C^\ast(\FD(C^{\ast,\ast}))\ra C^\ast(A, \Cb).$$

\medskip

  Let $\Kc$ be a sub Hopf algebra of $\Hc$.  Although $A$ is a
   $\Hc$-module algebra, the coalgebra
  $\Cc=C(\Hc,\Kc)=\Hc\ot_\Kc \Cb$
     does not inherit this property from $\Hc$ since
    the action of  $\Cc$ on $A$ is not well-defined.
    One cures this problem by letting $\Cc$ acts on the invariant
    sub algebra of $A$ under the action of $\Kc$. Let
  \begin{equation}
  A^\Kc=\{ a\in A\mid ka=\ve(k)a\}.
  \end{equation}
  One checks that the action of  $\Cc$ on $A^\Kc$ is
   well defined and satisfies \eqref{1}, \eqref{2}, and \eqref{3}.
  One notes that it is not possible to write the map \eqref{acpsi} for
   the case $A^\Kc$ and $\Cc$ because
  $A^\Kc$ is not a $\Hc$-module algebra. Instead one writes the invariant
  form of \eqref{acpsi} as follows.
  Let us introduce
  $$C^{p,q}_r=C^p_\Hc(A,M)\otb C^q_\Hc(\Hc,\Kc;M)$$ with its standard
  cyclic structure.
  Then one has a cyclic map
\begin{align}
&\Psi_r: \FD(C_r^{\ast,\ast})\ra C^\ast(A_\Kc),\\\notag
&\Psi_r(\phi\otb m\ot c^0\ot c^1\odots c^n)(a^0\odots a^n)=\\\notag
&\phi(m\ot c^0(a^0)\odots c^n(a^n)).
\end{align}
It is  shown  that the above map defines a cyclic map. Note that it
does not  land in $C^\ast(A)$.
\begin{corollary}
The map $\Psi$  induces the following maps   on cyclic cohomologies:
\begin{align}
&\Psi: HC^n( \FD(C^{\ast,\ast}))\ra HC^n(A)\\[.3cm]
&\Psi_c: HC^n( \FD(C_c^{\ast,\ast}))\ra HC^n(\Hom_\Hc(C,A))\\[.3cm]
&\Psi_r: HC^n( \FD(C_r^{\ast,\ast}))\ra HC^n(A_\Kc).
\end{align}
\end{corollary}
Now composing $\Psi$, $\Psi_c$, and $\Psi_r$ with  the corresponding
Alexander-Whitney map $AW$  one obtains  the following cup products:
\begin{align*}
&\cup=\Psi\circ AW :HC^p_\Hc(A,M)\ot HC^q_\Hc(C,M)\ra HC^{p+q}(A),\\[.3cm]
&\  \cup=\Psi_c\circ AW:HC^p_\Hc(A,M)\ot HC^q_\Hc(C,M)\ra
HC^{p+q}(\Hom_\Hc(C,A)),\\[.3cm]
&\cup=\Psi_r\circ AW:HC^p_\Hc(A, M)\ot HC^q_\Hc(\Hc,\Kc; M)\ra
HC^{p+q}(A^\Kc).
\end{align*}

\begin{proposition}\label{ccup}
The above cup product is precisely  given by the following formula
in the level of Hochschild cohomology.
\begin{align*}
&\cup:C^p_\Hc(A,M)\ot C^q_\Hc(C,M)\ra C^{p+q}(A),\\\notag &(\phi\cup
(m\ot c^0\odots c^q))(a^0\odots a^{p+q})=\\\notag &\phi(m\ot
c^0\ps{p+1}(a^0)c^1(a^1)\dots c^q(a^q)\ot c^0\ps{1} (a^{q+1})\odots
c^0\ps{p}(a^{p+q}))
\end{align*}
\end{proposition}
\begin{proof}
By composing the AW map given in \cite{lo} with $\Psi$ one obtains
the above formula.
\end{proof}

\section{ Module algebras paired with comodule algebras}
Let $\Hc$ be a Hopf algebra, $A$ a   left $\Hc$-module algebra,
 $B$ a
left $\Hc$-comodule algebra and $M$ be a right-left SAYD module
over $\Hc$.
One  constructs a crossed product algebra whose underlying vector
 space is $A\ot B$ with the $1\al 1$ as its unit and the following
 multiplication:
\begin{equation}
(a\al b)(a'\al b')= ab\ns{-1}(a')\al b\ns{0}b'
\end{equation}
Now consider  the two cocyclic modules
$$(C^\ast_\Hc(A,M), \d_i,\s_j,t), \quad \text{and}\quad
 (^\Hc C^\ast(B,M), d_i,s_j,\tau)$$
 introduced in \cite{hkrs2}  and let us make a new
bicocyclic module by just tensoring these two ones.  Its  $(p, q)$
component $C^{p,q}$ is given by
 $$\Hom_\Hc(M\ot A^{\ot p+1},\Cb)\ot ^\Hc \Hom(B^{\ot q+1}, M),$$
 with horizontal structure $\hd_i=d_i\ot\Id$,
$\hs_j=s_j\ot\Id$, and $\hta=t\ot \Id$ and vertical structure
$\vd_i=\d_i\ot\Id$, $\vs_j=\s_j\ot\Id$, and $\vta=\tau\ot \Id$.
Obviously $(C^{n,m},\hd,\hs,\hta,\vd,\vs,\vta)$  defines a
bicocyclic module. Now let us define  the following map
\begin{align}\label{abpsi}
&\Psi:C^{n,n}\ra \Hom((A\al B)^{\ot n+1},\Cb),\\\notag
& \Psi(\phi\ot\psi)(a^0\al b^0\odots a^n\al b^n)=\\\notag
&\phi(\psi(b^0\ns{0}\odots b^n\ns{0})\ot  S^{-1}(b^0\ns{1}\dots
 b^n\ns{-1})a^0\odots S^{-1}(b^n\ns{-n-1})a^n).
\end{align}
\begin{proposition}
The map $\Psi$ defines a  cyclic map between  the diagonal of
$C^{\ast,\ast}$ and the cocyclic module $C^\ast(A\al B)$.
\end{proposition}
\begin{proof}
We have to show that $\Psi$  commutes with the cyclic structures.
 We shall  check it for the first  face operator and the   cyclic
operator and leave the rest  to the reader.
\begin{align*}
&\Psi(\hd_0\vd_0 (\phi\ot\psi))(a^0\al b^0\odots a^{n+1}\al b^{n+1})=\\
&\Psi (d_0\phi\ot\d_0\psi))(a^0\al b^0\odots a^{n+1}\al b^{n+1})=\\
& \d_0\phi(\d_0\psi(b^0\ns{0}\odots b^{n+1}\ns{0})    \ot  S^{-1}
(b^0\ns{-1}\dots b^{n+1}\ns{-1})a^0\ot\dots\\
&\hspace{7.3cm}\dots\ot  S^{-1}(b^{n+1}\ns{-n-2})a^{n+1}  )=\\
& \phi(\psi(b^0\ns{0}b^1\ns{0}\odots b^{n+1}\ns{0})    \ot \\
&\hspace{2cm} \ot S^{-1}(b^0\ns{-1}\dots b^{n+1}\ns{-1})a^0S^{-1}
(b^1\ns{-2}\dots b^{n+1}\ns{-2})a^1\ot\dots\\
&\hspace{7.3cm}\dots\ot  S^{-1}(b^{n+1}\ns{-n-2})a^{n+1}  )=\\
& \phi(\psi(b^0\ns{0}b^1\ns{0}\odots b^{n+1}\ns{0})    \ot \\
&\hspace{2cm} \ot S^{-1}(b^0\ns{-1}\dots b^{n+1}\ns{-1})a^0S^{-1}
(b^1\ns{-2}\dots b^{n+1}\ns{-2})a^1\ot\dots\\
&\hspace{7.3cm}\dots\ot  S^{-1}(b^{n+1}\ns{-n-2})a^{n+1}  )=\\
& \phi(\psi(b^0\ns{0}b^1\ns{0}\odots b^{n+1}\ns{0}) \ot\\
&  \hspace{2cm}\ot S^{-1}(b^0\ns{-1}b^1\ns{-1}\dots b^{n+1}\ns{-1})
( a^0 b^0\ns{-2} a^1)\ot\dots\\
&\hspace{7.3cm}\dots\ot  S^{-1}(b^{n+1}\ns{-n-1})a^{n+1}  )=\\
&\Psi(\phi\ot\psi)(a^0 b^0\ns{-1}a^1\al b^0\ns{0}b^1\ot  a^2\al
 b^2\odots a^{n+1}\al b^{n+1})=\\
&d_0\Psi(\phi\ot \psi)(a^0\al b^0\odots a^{n+1}\al b^{n+1}).
\end{align*}
Using the facts that $\phi$ is $\Hc$ equivariant, $\psi$ is $\Hc$
colinear and $M$ is SAYD one has:
\begin{align*}
&\Psi(\hta\vta(\phi\ot\psi))(a^0\al b^0\odots a^{n}\al b^{n})= \\
&\Psi(t\phi\ot\tau\psi)(a^0\al b^0\odots a^{n}\al b^{n})=\\
&t\phi(\tau\psi(b^0\ns{0}\odots b^n\ns{0})\ot  S^{-1}(b^0\ns{1}\dots
 b^n\ns{-1})a^0\odots S^{-1}(b^n\ns{-n-1})a^n)=\\
&t\phi(\psi(b^n\ns{0}\ot b^0\ns{0}\odots b^{n-1}\ns{0}) b^n\ns{-1}\ot \\
&\hspace{2.5cm}\ot S^{-1}(b^0\ns{1}\dots b^{n-1}\ns{-1}b^n\ns{-2})a^0
\odots S^{-1}(b^n\ns{-n-2})a^n)=\\
&\phi([\psi(b^n\ns{0}\ot b^0\ns{0}\odots b^{n-1}\ns{0}) b^n\ns{-1}]
\ns{0}\ot \\
& S^{-1}([\psi(b^n\ns{0}\ot b^0\ns{0}\odots b^{n-1}\ns{0}) b^n\ns{-1}]
\ns{-1}(S^{-1}(b^n\ns{-n-2})a^n)\ot\\
&\ot S^{-1}(b^0\ns{1}\dots b^{n-1}\ns{-1}b^n\ns{-2})a^0\odots S^{-1}
(b^{n-1}\ns{-n+1}b^n\ns{-n-1})a^{n-1})=\\
&\phi(\psi(b^n\ns{0}\ot b^0\ns{0}\odots b^{n-1}\ns{0})b^n\ns{-3}\ot \\
& S^{-1}(S(b^n\ns{-2})b^n\ns{-1} b^0\ns{-1}\dots b^{n-1}\ns{-1})
b^n\ns{-4})(S^{-1}(b^n\ns{-n-2})a^n)\ot\\
&\ot S^{-1}(b^0\ns{1}\dots b^{n-1}\ns{-1}b^n\ns{-5})a^0\odots
S^{-1}(b^{n-1}\ns{-n+1}b^n\ns{-n-4})a^{n-1})=\\
&\phi(\psi(b^n\ns{0}\ot b^0\ns{0}\odots b^{n-1}\ns{0})b^n\ns{-1}\ot \\
& S^{-1}( b^0\ns{-1}\dots b^{n-1}\ns{-1}) b^n\ns{-2})
(S^{-1}(b^n\ns{-n-2})a^n)\ot\\
&\ot S^{-1}(b^0\ns{1}\dots b^{n-1}\ns{-1}b^n\ns{-3})
a^0\odots S^{-1}(b^{n-1}\ns{-n+1}b^n\ns{-n-2})a^{n-1})=\\
&\phi(\psi(b^n\ns{0}\ot b^0\ns{0}\odots b^{n-1}\ns{0})
\ot  S^{-1}( b^n\ns{-1}b^0\ns{-1}\dots b^{n-1}\ns{-1} )a^n)\ot\\
&\ot S^{-1}(b^0\ns{1}\dots b^{n-1}\ns{-1})a^0\odots
S^{-1}(b^{n-1}\ns{-n+1})a^{n-1})=\\
&\Psi(\phi\ot\psi)(a^n\al b^n \ot a^0\al b^0\odots
a^{n-1}\al b^{n-1})=\\
&t\Psi(\phi\ot\psi)(a^0\al b^0\odots a^{n}\al b^{n}).
\end{align*}
\end{proof}
\begin{corollary}
The map $\Psi$  defined in \eqref{abpsi} induces a map on cyclic
cohomologies:
\begin{align}
&\Psi: HC^n( \FD(C^{\ast,\ast}))\ra HC^n(A\al B).
\end{align}
\end{corollary}

 Now by composing $\Psi$   with  the corresponding
map $AW$ map one proves the existence  of the following map:
\begin{align}\label{abmap}
&\cup=\Psi\circ AW :HC^p_H(A,M)\ot \;\;^\Hc HC^q(B,M)\ra
HC^{p+q}(A\al B).
\end{align}
One uses  the formula of $AW$ map \cite{lo} to find the following
expression for the above cup product.
\begin{proposition}
The above cup product has the following formula in the level of
Hochschild cohomology.
\begin{align*}
&\phi\cup\psi(a^0\al b^0\odots a^{p+q}\al b^{p+q})=\\
&\phi(\psi(b^{q+1}\ns{0}\dots b^{p+q}\ns{0}b^0\ns{0}\ot b^1
\ns{0}\odots b^q\ns{0})\ns{-1}\ot \\
&  S^{-1}(b^0\ns{-1}\dots b^{q}\ns{-1})a^0\dots
S^{-1}(b^q\ns{-q-1})a^{q}\ot a^{q+1} \ot b^{q+1}\ns{-p-1}
 a^{q+2} \odots\\
& b^{q+1}\ns{-1}\dots b^{p+q-1}\ns{-1}a^{p+q}.
\end{align*}

\begin{example}{ \rm Let $G$ be a discrete group acting by unital
automorphisms on an algebra $A$ and let $k$ be a field of
characteristic zero. In \cite{kr2}, the Hopf cyclic cohomology
groups of the Hopf algebra $H=k G$ were computed in terms of group
cohomology with trivial coefficients:
$$^{k G}HC^p(k G, k)=\bigoplus_{i\geq 0} H^{p-2i} (G, k).$$
The cohomology groups $HC^q_{k G}(A,k)$ are easily seen to be the
cohomology of the subcomplex of {\it invariant cyclic cochains} on
$A$:
$$\varphi (ga_0, ga_1, \cdots, ga_n)=\varphi (a_0, a_1, \cdots, a_n),$$
for all $g \in G$ and $a_i \in A$. We denote this cohomology theory
by $HC^q_G(A)$. We have thus a pairing
$$H^p(G) \otimes HC^q_G(A) \longrightarrow HC^{p+q}(A\rtimes G).$$}
\end{example}

\end{proposition}
\section{ Cup product via traces}
In this section we derive some formulas for cup products defined in
\cite{kr1}.

Let us briefly recall it here. Let $A$ be a left $\Hc$ module
algebra, $B$  a left $\Hc$- comodule algebra and $M$ a SAYD module
on $\Hc$. Let
 also $\Om A$ be a DG $\Hc$-module algebra over $A$ and $\G B$ be a
 DG $\Hc$-comodule algebra over $B$.
 We recall that a closed $M$-trace on $\Om A$ is a linear map $\int: M\ot \Om A\ra
 \Cb$ such that
 \begin{align}
 &\int(h\ps{1}m\ot h\ps{2}\om)=\epsilon(h)\int(m\ot \om),\\
&\int( m\ot d\om)=0,\\
&\int(m\ot \om^1\ot \om^2)= (-1)^{\rm
deg(\om^1)deg(\om^2)}\int(m\sns{0}\ot S^{-1}(m\sns{-1})\om^2\om^1).
 \end{align}

Similarly a closed  $M$-trace on $\G B$ is defined as a linear map
$\int: \G B\ra M$ such that,
\begin{align}
& (\int \g)\sns{-1} \ot (\int \g)\sns{0}= \g\ns{-1}\ot \int(
\g\ns{0}),\\
& \int(d\g)=0,\\
 &\int(\g^1\g^2)=\int(\g^2\ns{0}\g^1)\g^2\ns{-1}.
\end{align}

 One identifies  closed cyclic cocycles  $\phi \in C^p_\Hc(A,M)$
 and $\psi\in ^\Hc HC^q(B,M)$
with  closed $M$-traces on $\Om(A)$ and $\G(B)$, the universal
$\Hc$-module DG algebra and $\Hc$-comodule algebra respectively, as
follows:
\begin{align}
\int_\phi m\ot a^0da^1\dots da^n=\phi (m\ot a^0\odots a^p)\\
\int_\psi(b^0db^1\dots db^n)= \psi(b^0\odots b^q)
\end{align}

Then one forms a DG algebra over $A\al B$ as the crossed product of
$\Om(A)$ and $\G(B)$, which we denote it by $\Om(A)\al\G(B)$. For
any  two closed $M$-traces  $\int_1$ and $\int_2$ on $\Om(A)$ and
$\G(B)$ one defines \cite{kr1} the closed trace  $\int_1\cup\int_2$
on $\Om(A)\al\G(B)$ by
\begin{equation}
(\int_1\cup\int_2)(\om\al \eta)= \int_1(\int_2(\om)\ot \eta),
\end{equation}

and hence the cup product of two cyclic cocycle is defined by
\begin{align*}
&(\phi\cup\psi)(a^0\al b^0\ot a^1\al b^1\odots a^{p+q}\al b^{p+q})=\\
&(\int_\phi\cup\int_\psi)(a^0\al b^0 d( a^1\al b^1)\dots
d(a^{p+q}\al b^{p+q})).
\end{align*}
Now we want to derive a formula for the above cup product. To this
end  we need to know the $(p,q)$ component of the form
 $$\theta^n=a^0\al b^0 d( a^1\al b^1)\dots d(a^{n}\al b^{n}).$$
 For  a permutation $\s\in Sh(q,p)$ we use $\bar\s(i)=\s(i)-1$, and
 $\hat\s(i)=\s(i)+1$, and define the following $p+q$ form
\begin{align*}
&\theta^n_\s= a^0 (b^0\dots b^{n-1})\ns{-n}a^1 \dots
(b^{\bar\s(q+1)-1}\dots
b^{n-1})\ns{-n+\s(q+1)-2}a^{\bar\s(q+1)}\\
&d((b^{\bar\s(q+1)}\dots b^{n-1})\ns{-n+\bar\s(q+1)}a^{\s(q+1)}\dots
(b^{\bar\s(q+2)-1}\dots
b^{n-1})\ns{-n+\s(q+2)-2}a^{\bar\s(q+2)})\dots\\
&\dots  d((b^{\bar\s(n)}\dots
b^{n-1})\ns{-n+\bar\s(n)}a^{\s(n)}\dots b^{n-1}\ns{-1}a^{n})\al\\
&b^0\ns{0}\dots b^{\bar\s(1)}\ns{0} d(b^{\s(1)}\ns{0}\dots
b^{\bar\s(2)}\ns{0})\dots d(b^{\s(q)}\ns{0}\dots b^{n-1}\ns{0}b^n).
\end{align*}
\begin{lemma}
The  $(q,p)$th component of the above form $\theta^{p+q}$ is given
by the following formula.
\begin{align*}
&\sum_{\s\in Sh(q,p)}(-1)^\s \theta^{p+q}_\s.
\end{align*}
\end{lemma}
\begin{proof}
We prove it by induction . Obviously it is true for $(p,q)=(0,0)$.
Assume that the lemma is true for all $(p,q)$ such that $p+q=n$, we
prove it for all $(p,q)$ that $p+q=n+1$.

The $(p,q)$th  component of $a^0\al b^0 d( a^1\al b^1)\dots
d(a^{p+q}\al b^{p+q})$ is

$\theta(da^{p+q}\al b^{p+q}) +\theta'(a^{p+q}\al db^{p+q})$, where
$\theta$ and $\theta'$ are $(p-1,q)$th, and $(p,q-1)$th component of
$a^0\al b^0 d( a^1\al b^1)\dots d(a^{p+q-1}\al b^{p+q-1})$
respectively. Now let $\mu\in Sh(q,p-1)$, one observes that
\begin{align*}
&\sum_{\s\in Sh(q-1,p)}(-1)^\s \theta^{p+q-1}_\s(a^{p+q}\al
db^{p+q})+\sum_{\s\in Sh(q,p-1)}(-1)^\s
\theta^{p+q-1}_\s(da^{p+q}\al b^{p+q})=\\
&\sum_{\s\in Sh(q,p)\quad \s(q)=p+q}(-1)^\s
\theta^{p+q}_\s+\sum_{\s\in Sh(q,p)\quad \s(p+q)=p+q }(-1)^\s
\theta^{p+q}_\s=\\
&\sum_{\s\in Sh(q,p)}(-1)^\s \theta^{p+q}_\s.
\end{align*}
\end{proof}

As a result one  has the following formula for cup product via
traces.
\begin{proposition}
 Let $\phi\in C^p_\Hc(A,M)$ and $\psi\in ~^\Hc C^q(B,M)$ respectively be two Hopf cyclic
  cocycles on $A$ and $B$ with coefficients   in a SAYD module  $M$.
 Then  $\phi\cup \psi\in C^{p+q}(A\rtimes B)$ is a cyclic cocycle and is
  precisely  given by the following formula
 \begin{align*}
&(\phi\cup\psi)(a^0\al b^0\ot a^1\al b^1\odots a^{p+q}\al b^{p+q})=\\
&\sum_{\s\in Sh(q,p)}(-1)^\s \p_{\bar\s(q)}\dots
\p_{\bar\s(1)}\phi(\p_{\bar\s(q+p)}\dots
\p_{\bar\s(q+1)}\psi(b^0\ns{0}\odots b^{p+q-1}\ns{0}b^{p+q})\ot\\
&\hspace{5cm} a^0\ot b^0\ns{-p-q}a^1\odots b^0\ns{-1}\dots
b^{p+q-1}\ns{-1}a^{p+q}).
\end{align*}
\end{proposition}

Similarly by following   \cite{kr1},  one uses cotraces on DG
coalgebras to defined a cup product that generalizes characteristic
map in Hopf cyclic cohomology. Indeed let $C$ be a $\Hc$-module
coalgebra and $A$ be a $\Hc$-module algebra satisfying the
conditions \eqref{0}\dots \eqref{3}. Let also $\phi\in C^q_\Hc(A,M)$
and
 $x:= m\ot_\Hc c^0\ot\dots\ot c^p \in C^p_\Hc(C,M)$ be   Hopf cyclic
cocycles. Then  one gets  the following  cyclic cocycle
${x\cup\phi}$ on $A$ by
\begin{align*}
&{x\cup\phi}(a_0\ot a_1\odots a_{p+q})=\\
 &\sum_{\s\in Sh(q,p)}(-1)^\s
\p_{\bar\s(q)}\dots \p_{\bar\s(1)}\phi(\p_{\bar\s(q+p)}\dots
\p_{\bar\s(q+1)}x(a_0\odots a_{p+q})),
\end{align*}
where $(m\ot c^0\odots c^n)(a^0\odots c^n):=m\ot c^0(a_0)\odots
c^n(a_n)$.

\end{document}